\documentclass[11pt]{amsart}
\usepackage{amsmath}
\usepackage{dsfont}

\theoremstyle{plain}
\newtheorem{theorem}{Theorem}[section]
\newtheorem{lemma}[theorem]{Lemma}

\newtheorem{corollary}[theorem]{Corollary}

\theoremstyle{definition}
\newtheorem{remark}[theorem]{Remark}

\numberwithin{equation}{section}

\def\be{\begin{equation}}
\def\ee{\end{equation}}

\begin{document}

\title[Minimal Graphs in the Hyperbolic Space]
{Minimal Graphs in the Hyperbolic Space\\ with Singular Asymptotic Boundaries}

\author{Qing Han}
\address{Beijing International Center for Mathematical Research\\
Peking University\\
Beijing, 100871, China} \email{qhan@math.pku.edu.cn}
\address{Department of Mathematics\\
University of Notre Dame\\
Notre Dame, IN 46556, USA} \email{qhan@nd.edu}
\author{Weiming Shen}
\address{School of Mathematical Sciences\\
Peking University\\
Beijing, 100871, China}  \email{wmshen@pku.edu.cn}
\author{Yue Wang}
\address{School of Mathematical Sciences\\
Peking University\\
Beijing, 100871, China}  \email{1201110027@pku.edu.cn}

\begin{abstract}
We study asymptotic behaviors of solutions $f$ to the Dirichlet problem
for minimal graphs in the hyperbolic space with singular asymptotic 
boundaries under the assumption that the boundaries are piecewise regular 
with positive curvatures. We derive an estimate of such solutions by the corresponding solutions 
in the intersections of interior tangent balls. The positivity of curvatures plays an important role. 
\end{abstract}

\thanks{The first author acknowledges the support of NSF
Grant DMS-1404596. The second author acknowledges the support of the 
graduate school of Peking University. 
The third author acknowledges the support of China Scholarship Council.}
\maketitle
\section{Introduction}

Assume that $\Omega\subset \mathbb{R}^{n}$ is a bounded domain.
Lin \cite{Lin1989Invent} studied the Dirichlet problem of the form
\begin{align}\label{eq-MinGmain}
\begin{split}
\Delta f-\frac{f_{i}f_{j}}{1+|\nabla f|^{2}}f_{ij}+\frac{n}{f}&=0 \quad\text{in }\Omega,\\
f&=0 \quad \text{on }\partial\Omega,\\
f&>0 \quad\text{in }  \Omega.
\end{split}
\end{align}
Geometrically, the graph of $f$ is a minimal surface in $\mathbb H^{n+1}$ with its asymptotic boundary at
infinity given by  $\partial \Omega$.
We note that the
equation in \eqref{eq-MinGmain} is a quasilinear non-uniformly elliptic
equation. It becomes singular on $\partial\Omega$ since $f = 0$ there.
Lin \cite{Lin1989Invent} proved that
\eqref{eq-MinGmain} admits a unique solution
$f\in C(\bar{\Omega})\bigcap C^{\infty}(\Omega)$
if
$\Omega\subset \mathbb{R}^{n}$
is a $C^2$-domain with a nonnegative boundary mean curvature $H_{\partial\Omega} \geq 0$
with respect to the inward
normal direction of $\partial\Omega$.
Concerning the higher global regularity, Lin proved
$f\in C^{1/2}(\bar{\Omega})$
if $H_{\partial\Omega} > 0$. 
In \cite{HanShenWang}, we proved that under the condition $H_{\partial\Omega} \geq 0$,
$f\in  C^{\frac{1}{n+1}}(\bar{\Omega})$
and
$$[f]_{C^{\frac{1}{n+1}}(\bar{\Omega})} \leq [(n+1) \operatorname{diameter}(\Omega)^{n}]^{\frac{1}{n+1}}.$$
This estimate does not depend on the regularity
of the domain, which allows us to discuss \eqref{eq-MinGmain} in domains with singularity.
In \cite{HanShenWang}, we also proved that \eqref{eq-MinGmain} admits a unique solution
$f\in C^{1/2}(\bar{\Omega})\bigcap C^{\infty}(\Omega)$
if $\Omega$ is a bounded domain
which is the intersection of finitely many bounded convex $C^{2}$-domains $\Omega_i$
with $H_{\partial\Omega_i} > 0$.

Concerning asymptotic behaviors of solutions $f$ of \eqref{eq-MinGmain}, we have the following result.  
Let $\Omega$ be a bounded $C^{2,\alpha}$-domain with $H_{\partial\Omega}>0$, 
for some $\alpha\in (0,1)$. 
Then, 
\begin{equation}\label{eq-C^2BasicExpansion}
\bigg|\left(\frac{H_{\partial\Omega}}{2d}\right)^{\frac12}f-1\bigg|\le Cd^{\frac{\alpha}2},\end{equation}
where $d$ is the distance function to $\partial\Omega$. 
We need to mention that the estimate \eqref{eq-C^2BasicExpansion} is sharp 
under the present regularity assumption. 
If $\partial\Omega$ has a higher 
regularity, we can expand more. For details, refer to \cite{HanJiang2014}. 

In this paper, we study asymptotic behaviors of solutions $f$ 
of \eqref{eq-MinGmain} near boundaries with singularity. 

There have been only a few results concerning the boundary behaviors of solutions of 
geometric PDEs in singular domains. This is partly due to 
the diversity of singularity and complexity of the relevant geometric problems.
The first two authors studied the asymptotic behaviors of solutions of the Liouville
equation in \cite{HanShen1} and solutions of the Loewner-Nirenberg problem in \cite{HanShen2}
in singular domains and proved that the solutions are well approximated by the
corresponding solutions in tangent cones at singular points on the boundary. 

Asymptotic behaviors of solutions of \eqref{eq-MinGmain} are more complicated than 
those of solutions of the Liouville equation and solutions of the Loewner-Nirenberg problem. 
As the estimate \eqref{eq-C^2BasicExpansion} illustrates, the positivity of the boundary 
mean curvature plays an important role in the estimates of solutions near $C^{2,\alpha}$-boundary. 
When we attempt to generalize \eqref{eq-C^2BasicExpansion}
to domains with singularity,  we cannot compare solutions $f$
of \eqref{eq-MinGmain} with the corresponding 
solutions in tangent cones if the tangent cones have zero mean curvature
wherever they are smooth. This is the case 
if the tangent cones are bounded by finitely many hyperplanes. 
We need a ``model" domain to preserve the positivity of the boundary mean curvature. 
Such a model domain is provided by the intersection of tangent balls at the singular points. 

We prove the following result for $n=2$. 

\begin{theorem}\label{main theorem2}
Let $\Omega$ be a bounded convex domain in $\mathbb R^2$
and, for $0\in\partial\Omega$ and some $R>0$, let 
$\partial\Omega\cap B_R$ consist of two $C^{2,\alpha}$-curves
$\sigma_1$ and $\sigma_2$
intersecting at the origin with an angle $\mu\pi$, for some
constants $\alpha, \mu\in (0,1)$, such that 
$\sigma_i$ has a positive curvature $\kappa_i$ at the origin, for 
$i=1,2$. Suppose  $f\in C(\bar\Omega)\cap C^\infty (\Omega)$ is the solution 
of \eqref{eq-MinGmain} and $h$ is the corresponding solution in 
$$\Omega_{\mu,\kappa_{1},\kappa_{2}}=
B_{\frac{1}{\kappa_{1}}} \left(\frac{1}{\kappa_{1}}\nu_1\right)
\bigcap B_{\frac{1}{\kappa_{2}}} \left(\frac{1}{\kappa_{2}}\nu_2\right),$$ 
where $\nu_i$ is the unit inner normal vector of $\sigma_i$ at the origin. 
Then, there exist a constant $r$ and a
$C^{2,\alpha}$-diffeomorphism $T$:
$B_r \rightarrow T(B_r)\subseteq\mathbb R^2$, with
$T(\Omega \bigcap B_r)= \Omega_{\mu,\kappa_{1},\kappa_{2}}\bigcap T(B_r)$ and
$T(\partial\Omega \bigcap B_r)= \partial \Omega_{\mu,\kappa_{1},\kappa_{2}}\bigcap T(B_r)$, 
such that,
for any $x\in B_{r}$,
\begin{align}
\label{main-estimate2} \left|\frac{f(x)}{h(Tx)}-1\right|\leq C|x|^{\beta},
\end{align}
where $\beta$ is a constant in $(0,\alpha/2]$ and $C$ is a positive 
constant depending only on $R$, $\alpha$, $\mu$, 
and the $C^{2,\alpha}$-norms of $\sigma_1$ and $\sigma_2$ in $B_R$. 
\end{theorem}

In the proof of Theorem \ref{main theorem2}, we will construct the map $T$, which is determined
by the distances to $\sigma_i$. We note that \eqref{main-estimate2} 
generalizes \eqref{eq-C^2BasicExpansion} to singular boundaries. 
We point out that if $\alpha$ is sufficiently small, we can take 
$\beta=\alpha/2$, which is optimal. 

We now describe briefly the proof of Theorem \ref{main theorem2}  is based on a
combination of isometric transforms and the maximum principle. Usually, when we discuss 
asymptotic behaviors of solutions $f$ in the domain $\Omega$ with a singularity at $x_0$, 
we compare such solutions 
with the corresponding solutions in tangent cones at $x_0$. However, the positivity of curvatures 
is not preserved for tangent cones. Instead, we use the solution
$f_{\mu,\kappa_{1},\kappa_{2}}$ in $\Omega_{\mu,\kappa_{1},\kappa_{2}}$
defined as intersections of 
tangent balls as stated in Theorem \ref{main theorem2}. 
Our goal is to compare the solution $f$ in $\Omega$ near $x_0$ 
with the solution
$f_{\mu,\kappa_{1},\kappa_{2}}$ in $\Omega_{\mu,\kappa_{1},\kappa_{2}}$. We note that 
a given point $x$ in $\Omega$ may not necessarily be a point in 
$\Omega_{\mu,\kappa_{1},\kappa_{2}}$. So as 
a part of the comparison of $f$ with $u_{V_{x_0}}$, we need to construct a map $T$, which maps 
$\Omega$ near $x_0$ onto $f_{\mu,\kappa_{1},\kappa_{2}}$ near $x_0$, and to compare $f(x)$ with 
$f_{\mu,\kappa_{1},\kappa_{2}}(Tx)$. We achieve this in two steps. 

In the first step, we construct two sets $\widetilde B$ and $\widehat B$ with the property 
$\widetilde B\subseteq\Omega\subseteq\widehat B$ near $x$. 
To construct such sets $\widetilde B$
and $\widehat B$, we first place two balls tangent to 
$\sigma_i$ at $p_i$, the closest point to $x$ on $\sigma_i$, for each $i=1, 2$. 
We can form 
$\widetilde B$ from the smaller balls and $\widehat B$ from the larger balls. 

In the second step, we compare the solution $f$ in $\Omega$ near $x_0$ with the solution 
$f_{\mu,\kappa_{1},\kappa_{2}}$ in $\Omega_{\mu,\kappa_{1},\kappa_{2}}$. 
To this end, we first compare $f$ with the solutions 
$\widetilde f$ and $\widehat f$ in $\widetilde B$ and $\widehat B$, respectively, and then compare 
$\widetilde f$ and $\widehat f$ with $f_{\mu,\kappa_{1},\kappa_{2}}$. 
We note that the sets $\widetilde B$, $\widehat B$, and $\Omega_{\mu,\kappa_{1},\kappa_{2}}$
have the same structure; namely, they are the intersections of two balls. 
Comparisons of solutions in these sets are aided by isometric transforms in the 
hyperbolic space. 

The paper is organized as follows. In Section \ref{sec-Existence},
we prove the existence of solutions of \eqref{eq-MinGmain}
in infinite cones and prove some basic
estimates for these solutions.
In Section \ref{sec-Localization}, 
we prove that asymptotic expansions near singular boundary points up to certain orders 
are local properties. In Section \ref{sec-C-2,alpha-boundary}, we study the asymptotic
expansions near singular points with positive curvatures and prove 
Theorem \ref{main theorem2}.

\section{Solutions in Cones}\label{sec-Existence}

In this section, we discuss \eqref{eq-MinGmain}
in infinite cones and prove the existence and uniqueness of its solutions.
We also derive some basic
estimates. Throughout this section, we assume $n=2$.

For some constant $\mu<1$, define
\begin{align}\label{eq-definition-cone}
V_{\mu}&=\{(r,\theta)\mid r\in(0,\infty),\theta\in(0,\mu\pi)\}.
\end{align}
This is an infinite cone in $\mathbb R^{2}$, expressed in polar coordinates.
Our goal is to find a solution $f$ of \eqref{eq-MinGmain} in the form
$$f=rh(\theta)\quad\text{in } \Omega=V_{\mu}.$$
By a straightforward calculation,
$\eqref{eq-MinGmain}$ has the form
\begin{align}
\label{po-MainEq}
\frac{h''+h}{r}-\frac{h'^{2}(h''+h)}{r(1+h^{2}+h'^{2})}+\frac{2}{rh}&=0.
\end{align}
In view of \eqref{po-MainEq}, we set
\begin{align}
Lh&=(h''+h)(1+h^{2})h+2(1+h^{2}+h'^{2}).
\end{align}
We note that $L$ is an operator acting on functions $h=h(\theta)$ on $(0,\mu\pi)$.

First, we construct supersolutions of $L$.

\begin{lemma}\label{lemma-supersol}
For some constant $\mu\in (0,1)$,
there exist constants  $A>0$, $B\ge 0$, $\alpha\ge 2$ and $\beta\in (0,1)$
such that
\begin{align}
L\left( A(\sin\frac{\theta}{\mu})^{\frac{1}{1+\alpha}}+ B(\sin\frac{\theta}{\mu})^{\frac{1}{1+\beta}}\right)
&\leq 0\quad\text{on }(0,\mu\pi).
\end{align}
\end{lemma}

\begin{proof}
For some $\alpha>0$, set
\begin{align}\label{def-x}
\varphi(\theta)=\left(\sin\frac{\theta}{\mu}\right)^{\frac{1}{1+\alpha}}.\end{align}
Then,
\begin{align*}
\varphi^{1+\alpha}=\sin\frac{\theta}{\mu}.
\end{align*}
A straightforward differentiation yields
\begin{align*}
\varphi'=\frac{\varphi^{-\alpha}}{1+\alpha}\frac{1}{\mu}\cos\frac{\theta}{\mu}
\end{align*}
and
$$\varphi''=-\frac{\varphi^{-\alpha}}{\mu^2(1+\alpha)}\sin\frac\theta\mu
-\frac{\alpha\varphi^{-\alpha-1}}{\mu(1+\alpha)}\varphi'\cos\frac\theta\mu.$$
By a simple substitution, we have
\begin{align*}
\varphi''&=-\frac{\varphi^{-\alpha}}{\mu^2(1+\alpha)}\varphi^{1+\alpha}
-\frac{\alpha\varphi^{-2\alpha-1}}{\mu^2(1+\alpha)^2}(1-\varphi^{2\alpha+2})\\
&=-\frac{1}{\mu^{2}(1+\alpha)^2}\varphi-\frac{\alpha}{\mu^2(1+\alpha)^2}\varphi^{-1-2\alpha}.
\end{align*}
Then, for some positive constant $A$,
\begin{align*}
L(A\varphi)&=A^{2}\varphi(1+A^{2}\varphi^{2})
\left[(1-\frac{1}{\mu^{2}(1+\alpha)^2})\varphi-\frac{2}{\mu^{2}(1+\alpha)^2}\varphi^{-2\alpha-1}\right]\\
&\qquad
+2\left[1+A^{2}\varphi^{2}+\frac{A^2}{\mu^2(1+\alpha^2)}\varphi^{-2\alpha}(1-\varphi^{2+2\alpha})\right].
\end{align*}

We first consider the case $\mu\le 1/3$.
With $\alpha=2$, we have
\begin{align*}
L(A\varphi)&=A^{2}\varphi(1+A^{2}\varphi^{2})
\left[(1-\frac{1}{9\mu^{2}})\varphi-\frac{2}{9\mu^{2}}\varphi^{-5}\right]
+2\left[1+A^{2}\varphi^{2}+A^{2}\frac{\varphi^{-4}}{9\mu^{2}}(1-\varphi^{6})\right]\\
&=A^{2}(3-\frac{3}{9\mu^{2}})\varphi^{2}+2
+A^{4}\varphi^{2}\left[(1-\frac{1}{9\mu^{2}})\varphi^{2}-\frac{2}{9\mu^{2}}\varphi^{-4}\right].\end{align*}
Hence,
\begin{align*}
L(\sqrt{3\mu} \varphi)\leq 2-2\varphi^{-2}\leq 0.
\end{align*}

Next, we consider the case $\mu>1/3$.
For some positive $\alpha$ and $\beta$, set $\varphi$ as in \eqref{def-x}
and
$$\psi(\theta)=\left(\sin\frac{\theta}{\mu}\right)^{\frac{1}{1+\beta}}.$$
Then,
$$\psi=\varphi^{\frac{1+\alpha}{1+\beta}}\leq \varphi.$$
For some positive constants $A$ and $B$, set
$$h=A\varphi+B\psi.$$
Now, we write
\begin{align}\label{Lf}Lh=I+II,\end{align}
where
\begin{align*}I&=[1+A^{2}\varphi^{2}+B^{2}\psi^{2}+2AB\varphi\psi]
    [A\varphi+B\psi]\\
&\qquad\cdot\bigg\{A\left[(1-\frac{1}{\mu^{2}(1+\alpha)^{2}})\varphi
    -\frac{\alpha}{\mu^{2}(1+\alpha)^{2}}\varphi^{-1-2\alpha}\right]\\
    &\qquad\quad+B\left[(1-\frac{1}{\mu^{2}(1+\beta)^{2}})\psi
    -\frac{\beta}{\mu^{2}(1+\beta)^{2}}\psi^{-1-2\beta}\right]\bigg\},\end{align*}
and
\begin{align*}
II&=2\bigg[1+A^{2}\varphi^{2}+B^{2}\psi^{2}+2AB\varphi\psi\\
&\qquad+A^{2}\frac{\varphi^{-2\alpha}}
    {\mu^2(1+\alpha)^2}(1-\varphi^{2+2\alpha})+B^{2}\frac{\psi^{-2\beta}}
    {\mu^2(1+\beta)^2}(1-\psi^{2+2\beta})\\
&\qquad +2AB\frac{\varphi^{-\alpha}}
    {\mu(1+\alpha)}\frac{\psi^{-\beta}}{\mu(1+\beta)}\sqrt{(1-\varphi^{2+2\alpha})
    (1-\psi^{2+2\beta})}\bigg].
\end{align*}
Fix an $\alpha\in(2,+\infty)$ and take
$$\beta=\text{min}\left\{\frac{1}{2}\left(\frac{1}{\mu}-1\right),\frac{1}{100}\right\}.$$
With $\psi\le 1$, it is easy to check
\begin{equation}\label{eq-I-beta}\left(1-\frac{1}{\mu^{2}(1+\beta)^{2}}\right)\psi
    -\frac{\beta}{\mu^{2}(1+\beta)^{2}}\psi^{-1-2\beta}<0.\end{equation}
In fact, we only need to require $\beta<1/\mu-1$.
Next, if $\sin\frac{\theta}{\mu}\leq\frac{1}{1+\alpha},$ i.e.,
$$\varphi\leq \left(\frac{1}{1+\alpha}\right)^{\frac{1}{1+\alpha}},$$
we have
\begin{equation}\label{eq-I-alpha}\left(1-\frac{1}{\mu^{2}(1+\alpha)^{2}}\right)\varphi
    -\frac{\alpha}{\mu^{2}(1+\alpha)^{2}}\varphi^{-1-2\alpha}<0.\end{equation}
For $\sin\frac{\theta}{\mu}\in[\frac{1}{1+\alpha},1],$ we have 
$\psi\ge (\frac{1}{1+\alpha})^{\frac1{1+\beta}}$ and
hence
\begin{align*}
&\left(1-\frac{1}{\mu^{2}(1+\alpha)^{2}}\right)\varphi
    -\frac{\alpha}{\mu^{2}(1+\alpha)^{2}}\varphi^{-1-2\alpha}\\
    &\qquad\quad+C\left[\left(1-\frac{1}{\mu^{2}(1+\beta)^{2}}\right)\psi
    -\frac{\beta}{\mu^{2}(1+\beta)^{2}}\psi^{-1-2\beta}\right]\\
&\quad\leq1-C\left(\frac{1}{\mu^{2}(1+\beta)^{2}}-1\right)\psi\le -1,
\end{align*}
by choosing $C>0$ large.
By combining with \eqref{eq-I-beta} and \eqref{eq-I-alpha}, we have, on $[0,\mu\pi]$,
\begin{align}\label{eq-I}\begin{split}
&\left(1-\frac{1}{\mu^{2}(1+\alpha)^{2}}\right)\varphi
    -\frac{\alpha}{\mu^{2}(1+\alpha)^{2}}\varphi^{-1-2\alpha}\\
    &\qquad+C\left[\left(1-\frac{1}{\mu^{2}(1+\beta)^{2}}\right)\psi
    -\frac{\beta}{\mu^{2}(1+\beta)^{2}}\psi^{-1-2\beta}\right]\le -\eta,\end{split}\end{align}
for some positive constant $\eta$.  With $A$ to be determined, we set
$$B=CA.$$
With such a choice of $B$, we proceed to prove $Lh\le 0$ for suitably chosen $A$.

We first consider
$\sin\frac{\theta}{\mu}<\frac{1}{1+\alpha}.$ By \eqref{eq-I-beta} and \eqref{eq-I-alpha}, we have
\begin{align*}I&\le1\cdot (A\varphi)\cdot A\left[\left(1-\frac{1}{\mu^{2}(1+\alpha)^{2}}\right)\varphi
    -\frac{\alpha}{\mu^{2}(1+\alpha)^{2}}\varphi^{-1-2\alpha}\right]\\
&\le -\frac{A^2\alpha}{\mu^{2}(1+\alpha)^{2}}\varphi^{-2\alpha}+C_1A^2\varphi^{-2\alpha+\tau},
\end{align*}
for some positive constant $\tau$. We note that the omitted terms in $I$ are all nonpositive.
Similarly, we have
$$II\le \frac{2A^2}{\mu^{2}(1+\alpha)^{2}}\varphi^{-2\alpha}+2+C_2A^2\varphi^{-2\alpha+\tau}.$$
Hence, by \eqref{Lf}, 
\begin{align*}
Lh\leq A^{2}\frac{2-\alpha}{\mu^{2}(1+\alpha)^{2}}\varphi^{-2\alpha}
+2+C_0A^2\varphi^{-2\alpha+\tau}.
\end{align*}
In the following, we always choose $A\ge 1$.
There exists a small $\delta$, independent of $A$, such that
$$Lh<0 \quad\text{if }\sin\frac{\theta}{\mu}\le\delta.$$
For $\sin\frac{\theta}{\mu}>\delta,$
we have, by \eqref{eq-I},
\begin{align*}
I\le -(A^2\varphi^2)\cdot(A\varphi)\cdot A\eta=-A^4\varphi^3\eta\le -A^4\eta \delta^{\frac{3}{1+\alpha}}.\end{align*}
On the other hand,
$$II\le C_*A^2.$$
Hence, by choosing $A$ sufficiently large, depending on $\delta$, we have
$$Lh=I+II<0\quad\text{if }\sin\frac{\theta}{\mu}\ge\delta.$$
We have the desired result.
\end{proof}

\begin{remark}
The supersolution $h$ in Lemma \ref{lemma-supersol} satisfies
$$h \leq C_{\alpha,\mu}(\sin\frac{\theta}{\mu})^{\frac{1}{1+\alpha}}.$$
In particular, for $\mu \leq 1/3$, we can take $\alpha=2$ and
$C_{2,\mu}\to0$ as $\mu\to0.$
\end{remark}

Next, we introduce an important transform.
For any $L>0,$ we define the operator $T_L$ by
\begin{align}\label{eq-T}
T_L(x_1,x_2,x_3)=\frac{L}{(x_1-L)^2+x_2^2+x_3^2}(L^2-x_1^2-x_2^2-x_3^2
,2Lx_2,2Lx_3).
\end{align}
Then, $T_L$ is an isometric automorphism on $\mathds{H}^3$ and,
restricted on $\mathbb R^{2}\times{\{x_3=0\}},$
$T_L$ is a conformal transform, which maps the point $(L,0,0)$ to infinity.
In fact, $T_L(x_1,x_2,x_3)$ is a composition of the following transformations.

First, consider
$$G_{1}: (x_1,x_2,x_3)\mapsto (x_1,x_2,x_3+L).$$
Then, $G_1$ maps $\{(x_1,x_2,x_3)|\,x_3>0\}$ to $\{(x_1,x_2,x_3)|\,x_3>L\}$. Second, consider
$$G_{2}: (x_1,x_2,x_3)\mapsto \frac{2L^{2}}{(x_1^2+x_2^2+x_3^2)}(x_1,x_2,x_3).$$
Then, $G_2$ maps $\{(x_1,x_2,x_3)|\,x_3>L\}$ to $\{(x_1,x_2,x_3)|\,x_1^2+x_2^2+(x_3-L)^2< L^2\}$
and maps $\{(x_1,x_2,x_3)|\,x_1^2+x_2^2+(x_3-L)^2< L^2\}$ to $\{(x_1,x_2,x_3)|\,x_3>L\}$.
Next, consider
$$G_{3}: (x_1,x_2,x_3)\mapsto (x_1,x_2,x_3-L).$$
Then, $G_3$ maps $\{(x_1,x_2,x_3)|\,x_3>L\}$ to $\{(x_1,x_2,x_3)|\,x_3>0\}$.
Last, consider $$G_{4}: (x_1,x_2,x_3)\mapsto (x_3,x_2,-x_1).$$
Then, $G_4$ is an orthogonal transform which
rotates the $x_1x_{3}$-plane by ${\pi}/{2}$ clockwisely.
Then,
$$T_L=G_3G_2 G_1G_4G_3G_{2}G_1.$$
It is easy to see
$$\frac{2L^2 x_3}{(x_1-L)^2+x_2^2+x_3^2}\rightarrow 0
\quad\text{as }x_1^2+x_2^2+x_3^2\rightarrow\infty.$$

Now we proceed to prove the existence and  uniqueness of solutions of \eqref{eq-MinGmain}
in cones. 

\begin{theorem}\label{thrm-Existence}
Let $V_\mu$ the the cone as in \eqref{eq-definition-cone}, for some $\mu\in (0,1)$. 
Then, there exists a unique solution $f$ of \eqref{eq-MinGmain} for
$\Omega=V_{\mu}$. Moreover, $f$ has the form $rh(\theta)$.
\end{theorem}

\begin{proof}
We first prove the existence.
For any $R>0$, set \begin{align*}
V_{\mu,R}&=\{(r,\theta)|\, r\in(0,R),\theta\in(0,\mu\pi)\}.
\end{align*}
By Theorem 3.1 \cite{HanShenWang}, there exists a
unique solution $f_{\mu,R}$ of \eqref{eq-MinGmain} for $\Omega=V_{\mu,R}$
and, by the maximum principle and Lemma \ref{lemma-supersol},
\begin{align}\label{eq-estimate-f}
f_{\mu,R}\leq  C_{\alpha,\mu }r\left(\sin\frac{\theta}{\mu}\right)^{\frac{1}{1+\alpha}}
\quad\text{in }V_{\mu, R}.
\end{align}
By the maximum principle again, we have,
for any positive $R_{1}$ and $R_{2}$ with $R_{1}<R_{2}$, 
\begin{align}\label{eq-increasing-f}
f_{\mu,R_{1}}\leq f_{\mu,R_{2}}
\quad\text{in }V_{\mu, R_1}.
\end{align}
Next, the uniqueness and scaling imply
$$f_{\mu,k}(x)=kf_{\mu,1}\left(\frac{x}{k}\right).$$
For any positive $\delta$ sufficiently small, set
\begin{align*}
    W_{\mu,\delta}=\{(r,\theta)\in V_{\mu}|r\ \in (0,\infty),\theta\in (\delta\pi,(\mu-\delta)\pi)\}.
\end{align*}
Then for any $x\in W_{\mu,\delta},$ we have 
$$d\geq (\sin\delta\pi) r.$$ 
By employing the method in the proof of Theorem 3.2 in \cite{HanShenWang},
we can prove, for any $x\in W_{\mu,\delta}$ and any $k \geq 2[|x|+1]$, 
\begin{align}\label{fk grad estimate}
    |\nabla f_{\mu, k}|\leq C(\mu,\delta, x).
\end{align}
In fact, for any $k \geq 2[|x|+1]$, $ B_{{d_x}/{4}}(x) \subseteq V_{\mu,k}$. 
By Step 1 in the proof of Theorem 3.2 in \cite{HanShenWang}, we have
\begin{align}\label{fk int estimate}
    \|\nabla f_{\mu, k}\|_{L^p(B_{{d_x}/{4}}(x))}\leq C(p,\mu,\delta)d^\frac{n}{p}.
\end{align}
Note that ${f_{\mu, k}}\geq{f_{\mu, 2[|x|+1]}}$ for any $k \geq 2[|x|+1]$, by  \eqref{eq-increasing-f}.
By combining with \eqref{fk int estimate} and proceeding as in 
Step 2 in the proof of Theorem 3.2  in \cite{HanShenWang}, we can get 
\eqref{fk grad estimate} by applying the $W^{2,p}$-estimate. 
Now for any $x\in V_{\mu},$ there exists a positive small $\delta$ such that
$x\in W_{\mu,\delta}.$ Then for any $y\in B_{{d_x}/{4}}(x)$ and large enough positive $k,$ we have
$$|f_{\mu,k}(x)-f_{\mu,k}(y)|
\leq C(\mu,\delta,x)|x-y|.$$
Therefore, by \eqref{eq-increasing-f} and the interior estimate,
for any $x\in V_{\mu},$ we have that $f_{\mu,R}(x)$ converges to some $f_{\mu}(x)$
as $R\to \infty$ and  $f_{\mu}\in C^{\infty}(V_{\mu})$ is a solution of
\begin{equation}\Delta f-\frac{f_{i}f_{j}}{1+|\nabla f|^{2}}f_{ij}+\frac{2}{f} =0 \quad\text{in }V_{\mu}.
\end{equation}
By \eqref{eq-estimate-f}, $f_\mu$ is continuous up to the boundary of $V_\mu$ and
$f_{\mu}=0$ on $\partial V_\mu$. In summary, $f_\mu$ is a solution of
\eqref{eq-MinGmain} for
$\Omega=V_{\mu}$.
Moreover, for any positive integer $k,$
$f_{\mu,kR}(x)$ converges to $f_{\mu}(x)$ as $R$ goes to infinity. By the property
\begin{align*}
f_{\mu,kR}(x)=kf_{\mu,R}\left(\frac{x}{k}\right),\end{align*}
we obtain
\begin{align*}f_{\mu}(x)=kf_{\mu}\left(\frac{x}{k}\right).\end{align*}
Therefore, we can write  $f_{\mu}=rh(\theta)$ for some function $h=h(\theta)$ on
$(0,\mu\pi)$.

We now prove the uniqueness.
For convenience, we rotate $\mathbb R^2$ and assume
\begin{align}
V_{\mu}&=\left\{(r,\theta)|\, r\in(0,\infty),\theta\in\left(-\frac{\mu\pi}{2},\frac{\mu\pi}{2}\right)\right\}.
\end{align}
Let $f_1$ and $f_2$ be two solutions of \eqref{eq-MinGmain} for $\Omega=V_{\mu}.$
By Remark 2.3 \cite{Lin1989Invent}, $\{(x,f_1(x))\}$ and
$\{(x,f_2(x))\}$ are two absolutely area-minimizing hypersurfaces with
the asymptotic boundary $\partial V_\mu$. Let $T_1$ be the map defined in \eqref{eq-T} with
$L=1$. Then,
$T_1^{-1}|\{x_3=0\}$ maps $V_{\mu}$ conformally to
$$\widetilde \Omega=
B_{\frac{1}{\sin\frac{\mu\pi}{2}}}\left(\left(0,-\cot\frac{\mu\pi}{2}\right)\right)
\bigcap B_{\frac{1}{\sin\frac{\mu\pi}{2}}}\left(\left(0,\cot\frac{\mu\pi}{2}\right)\right),$$
and maps the absolutely area-minimizing hypersurface
$\{(x,f_i(x))\}$ with the asymptotic boundary $\partial V_{\mu}$
to the absolutely area-minimizing hypersurface
$\{(y,\widetilde f_{i}(y))\}$
with the asymptotic boundary
$\partial \widetilde\Omega$, $i=1,2$.
By Corollary 2.4 \cite{Lin1989Invent}, $\widetilde f_{1}=\widetilde f_2$. Hence, $f_1= f_2$.
\end{proof}

Next, we proceed as Lin \cite{Lin1989Invent}. Let $f$ be a solution
of \eqref{eq-MinGmain} in $\Omega$.
Locally near each boundary point,
the graph of $f$ can be represented by a function
over its vertical tangent plane. Specifically,
we fix a boundary point of $\Omega$, say the origin, and assume that
the vector $e_n=(0,\cdots, 0,1)$ is the interior normal vector to $\partial\Omega$
at the origin. Then, with $x=(x',x_n)$, the $x'$-hyperplane is the tangent plane of
$\partial\Omega$ at the origin and the boundary $\partial\Omega$ can be expressed
in a neighborhood of the origin as a graph of a smooth function over $\mathbb R^{n-1}\times\{0\}$,
say
$$x_n=\varphi(x').$$
We now denote points in $\mathbb R^{n+1}=
\mathbb R^n\times\mathbb R$ by $(x',x_n,t)$. The vertical hyperplane
given by $x_n=0$ is the tangent plane to the graph of $f$ at the origin in $\mathbb R^{n+1}$.
We can represent the graph of $f$ as a graph of a new function $u$ defined in terms of
$(x', 0, t)$ for small $x'$ and $t$, with $t>0$. In other words, we treat
$\mathbb R^n=\mathbb R^{n-1}\times\{0\}\times\mathbb R$ as our new base space and write
$u=u(x', t)$. Then, for some $R>0$,
$u$ satisfies
\begin{align}\label{eq-Intro-Equ}
\Delta u - \frac{u_i u_j}{1+|\nabla u|^2}u_{ij}-\frac{n u_{t}}{t}=0  \quad \text{in } B_R^+
\end{align}
and
\begin{align}\label{eq-Intro-EquCondition}
u(\cdot, 0)=\varphi\quad\text{on }B_R'.
\end{align}
We note that $u$ and $f$ are related by
\begin{align}\label{eq-u}
x_{n}=u(x',t)
\end{align}
and
\begin{align}\label{eq-f}
t=f(x',x_n).
\end{align}
Set
\begin{align}\label{b1a-v}
u_{n+1}(x',t)=\varphi(x')+\sum_{i=2}^{n+1}c_i(x')t^i+c_{n+1,1}(x')t^{n+1}\log t.
\end{align}
 We have the following result.

\begin{lemma}\label{lemma-Expension-u}
For some constant $\alpha\in (0,1)$,
let $\varphi\in C^{n+1,\alpha}(B'_R)$ be a given function
and $u\in C(\bar B^+_R)\cap C^\infty(B^+_R)$ be
a solution of \eqref{eq-Intro-Equ}-\eqref{eq-Intro-EquCondition}. Then,
there exist functions  $c_i\in C^{n+1-i, \epsilon}(B_R')$, for $i=0, 2, 4, \cdots, n+1$,
$c_{n+1,1}\in C^{\epsilon}(B_R')$,
and any $\epsilon\in (0,\alpha)$,
such that,
for $u_{n+1}$ defined as in \eqref{b1a-v},
for any $m=0, 1, \cdots, n+1$, any $\epsilon\in (0,\alpha)$, and any $r\in (0, R)$,
\begin{equation*}
\partial_t^m (u-u_{n+1})\in C^{\epsilon}(\bar B^+_r),
\end{equation*}
and, for any $(x',t)\in
B^+_{R/2}$,
\begin{equation*}
|\partial_t^m (u-u_{n+1})(x',t)|
\le C t^{n+1-m+\alpha},
\end{equation*}
for some positive constant $C$ depending only on $n$,  $\alpha$, $R$,
the $L^\infty$-norm of $u$ in $B_R^+$, and the $C^{n+1, \alpha}$-norm  of $\varphi$ in
$B_R'$.
\end{lemma}

Lemma \ref{lemma-Expension-u} follows from Theorem 1.1 \cite{HanJiang2014}
by taking $\ell=k=n+1$. In fact, $c_2, \cdots, c_n$ and $c_{n+1,1}$ are coefficients
for local terms and have explicit expressions in terms of $\varphi$. Meanwhile,
$c_{n+1}$ is the coefficient of the first nonlocal term.

\begin{corollary}\label{cor-Existence}
Let $V_\mu$ be the cone as in \eqref{eq-definition-cone}, 
for some $\mu\in (0,1)$, and let $f=rh(\theta)$ be the solution of \eqref{eq-MinGmain} for
$\Omega=V_{\mu}$ as in Theorem \ref{thrm-Existence}.
Then,
$$\lim_{\theta\to0}\theta^{-1/3}h(\theta)>0,$$
and
$$\lim_{\theta\to0}\theta^{2/3}h'(\theta)>0.$$
\end{corollary}

\begin{proof} Take $n=2$ and consider \eqref{b1a-v} at $(r,\theta)=(1,0)\in \partial V_\mu$. 
Then, $c_2=0$, $c_{3,1}=0$, and, by renaming the coefficient for $t^3$, 
\begin{align}\label{u exp}
    u=a_\mu t^3+O(t^4),
\end{align}
where $a_\mu$ is a constant and 
\begin{align}\label{2}
\begin{split}
    u&=\tan\theta,\\
    t&=f=\frac{1}{\cos \theta}h(\theta).
\end{split}
\end{align}
We write the coefficient of $t^3$ as $a_\mu$ to emphasize its dependence on $\mu$. 

Fix an arbitrary $\alpha\in (2,\infty)$. By Lemma \ref{lemma-supersol}, we have
$$f\leq C_{\alpha,\mu}r\left(\sin \frac{\theta}{\mu}\right)^{\frac{1}{1+\alpha}}\quad\text{in }V_\mu.$$
Therefore, if $t>0$ is small, we have
$$u\geq \widetilde C_\alpha t^{1+\alpha},$$
for some positive constant $\widetilde C_\alpha$. With \eqref{u exp}, this implies
$a_\mu>0$ and $u\sim t^3$.
Therefore, for $\theta>0$ small, we have
$$f\leq Cr\left(\sin\frac{\theta}{\mu}\right)^{\frac{1}{3}},$$
and
\begin{equation}\label{g-growth rate}
h(\theta)=\frac{1}{\sqrt[3]{a_\mu}} \theta^{\frac{1}{3}}+O(\theta^{\frac23}).
\end{equation}
Next, we note
\begin{align*}
    u_t=3a_\mu t^2+O(t^3).
\end{align*}
By \eqref{2} and $u_\theta=u_tt_\theta,$
we have
\begin{align*}
    \frac{1}{\cos^2\theta}&=u_t
    \left(\frac{\sin \theta h(\theta)}{\cos ^2 \theta}+\frac{h'(\theta)}{\cos\theta}\right).
\end{align*}
With \eqref{g-growth rate}, we have
\begin{align}\label{deriv-g-growth rate}
h'(\theta)=\frac{1}{3\sqrt[3]{a_\mu}} \theta^{-\frac{2}{3}}+O(\theta^{-\frac13}).
\end{align}
We have the desired results.
\end{proof}

In Theorem \ref{thrm-Existence}, we proved the existence of solutions of $\eqref{po-MainEq}$
in $V_{\mu}$ and obtained the unique solutions of the form $f=rh_{\mu}(\theta).$
Here, $h_{\mu}(\theta)$ is a function of $\theta$ on $(0,\mu\pi)$ and
we adopt the subscript $\mu$ to indicate that $h_{\mu}$.

We now compare $h_{\mu}$ for different $\mu$.

\begin{lemma}\label{lem-sol-diff-cone} Let $\mu_1$ and $\mu_2$ be two distinct constants in $(0,1)$
and $h_{\mu_{i}}(\theta)$ be the solution of \eqref{po-MainEq} on
$(0,\mu_i\pi)$, for $i=1,2.$ 
Then, for $\mu_1<\mu_{2}<\mu_{1}+\delta(\mu_{1}),$
\begin{align}\label{eq-sol-diff-cone}
h_{\mu_{1}}(\mu_{1}\theta)\le h_{\mu_{2}}(\mu_{2}\theta)
\le C_{\mu_{1},\mu_{2}}h_{\mu_{1}}(\mu_{1}\theta)
\quad\text{for any } \theta\in (0,\pi),
\end{align}
where $\delta(\mu_1)$ and $C_{\mu_1, \mu_2}$ are positive constants given by
\begin{align*}
    \delta(\mu_{1})=\bigg(\left(\frac{1}{8b_{\mu_{1}}}+1\right)^{\frac12}-1\bigg)\mu_{1},
\end{align*}
and
\begin{align}\label{C12}
 C_{\mu_{1},\mu_{2}}=\left(1+\frac{b_{\mu_{1}}}{\mu_{1}^{2}}(\mu_{2}^{2}-\mu_{1}^{2})\right)^{\frac{1}{2}},
\end{align}
with
\begin{align}\label{eq-def-b}
    b_{\mu_{1}}=\max\left\{\frac{81}{128}\sup_{\theta\in(0,\mu_1 \pi)}g_{\mu_{1}}^4,
    \sup_{\theta\in(0,\mu_1 \pi)}\,\left[\frac{3(g_{\mu_{1}})^2+2}{-((g_{\mu_{1}})''(g_{\mu_{1}})^3+
    (g_{\mu_{1}})^4)}\right]\right\}. 
\end{align}
\end{lemma}

\begin{proof}
By \eqref{po-MainEq}, we have 
\begin{align}\label{eq-orgn}
h_{\mu_{i}}''h_{\mu_{i}}+h_{\mu_{i}}''h_{\mu_{i}}^{3}
+3h_{\mu_{i}}^2+h_{\mu_{i}}^4+2+2h_{\mu_{i}}'^{2}=0\quad\text{on }(0,\mu_i\pi).
\end{align}
For convenience, we set
\begin{align*}
\widetilde{h}_{\mu_{i}}(\theta)=h_{\mu_{i}}(\mu_{i}\theta)
\quad\text{for } \theta\in (0,\pi).
\end{align*}
Note
\begin{align*}
\widetilde{h}_{\mu_{i}}'(\theta)=\mu_ih_{\mu_{i}}'(\mu_{i}\theta),\quad
\widetilde{h}_{\mu_{i}}''(\theta)=\mu_i^2h_{\mu_{i}}''(\mu_{i}\theta)
\quad\text{for } \theta\in (0,\pi).
\end{align*}
Then, (\ref{eq-orgn}) implies
\begin{align*}
\frac{1}{\mu_{i}^2}\widetilde{h}_{\mu_{i}}''\widetilde{h}_{\mu_{i}}
+\frac{1}{\mu_{i}^2}\widetilde{h}_{\mu_{i}}''\widetilde{h}_{\mu_{i}}^{3}
+3\widetilde{h}_{\mu_{i}}^2+\widetilde{h}_{\mu_{i}}^4+2+\frac{2}{\mu_{i}^2}\widetilde{h}_{\mu_{i}}'^{2}=0
\quad\text{on }(0,\pi).\end{align*}
In view of this equation, we set
\begin{align*}
\mathfrak{L}_{\mu_{2}}h=\frac{1}{\mu_{2}^2}h''h+\frac{1}{\mu_{2}^2}h''h^{3}
+3h^2+h^4+2+\frac{2}{\mu_{2}^{2}}h'^{2}.\end{align*}

We now prove the second inequality in (\ref{eq-sol-diff-cone}).
We claim, for the positive constant $C$ as in \eqref{C12},
\begin{equation}\label{eq-claim}
\mathfrak{L}_{\mu_{2}}(C\widetilde{h}_{\mu_{1}})\leq 0\quad\text{on }(0,\pi).\end{equation}
Assuming \eqref{eq-claim}, we proceed as follows. Set $Q$ by
$$Q(h)=
\Delta h-\frac{h_{i}h_{j}h_{ij}}{1+|\nabla h|^{2}}+\frac{2}{h}.$$
Comparing $Q$ and $\mathfrak{L}_{\mu_{2}},$ we note that
\eqref{eq-claim} implies
$$Q\left(Crh_{\mu_{1}}\left(\frac{\mu_{1}}{\mu_{2}}\theta\right)\right)\leq 0\quad\text{in }V_{\mu_2}.$$
As shown in the proof of Theorem \ref{thrm-Existence},
we can take a sequence $f_{\mu_{2},k}$ such that 
$$Q(f_{\mu_{2},k})=0\quad\text{in }V_{\mu_{2},k} ,$$ and
$$f_{\mu_{2},k}=0\quad\text{on }\partial V_{\mu_{2},k}. $$
Then, $f_{\mu_2, k}\to f_{\mu_{2}}=rh_{\mu_{2}}(\theta)$ as $k\to \infty$ and, 
by the maximum principle,
$$f_{\mu_{2},k}\leq Crh_{\mu_{1}}\left(\frac{\mu_{1}}{\mu_{2}}\theta\right)\quad\text{in }V_{\mu_{2},k}.$$
Letting $k\to\infty$, we obtain
$$rh_{\mu_{2}}(\theta)=f_{\mu_{2}}\leq Crh_{\mu_{1}}\left(\frac{\mu_{1}}{\mu_{2}}\theta\right)
\quad\text{in }V_{\mu_{2}}.$$
This is the desired conclusion.

Now, we proceed to prove \eqref{eq-claim}.
Note
\begin{align*}
\mathfrak{L}_{\mu_{2}}(C\widetilde{h}_{\mu_{1}})(\theta)&
=\bigg[C^2\frac{\mu_{1}^2}{\mu_{2}^2}h_{\mu_{1}}''h_{\mu_{1}}
   +C^4\frac{\mu_{1}^2}{\mu_{2}^2}h_{\mu_{1}}''hg_{\mu_{1}}^{3}\\
&\qquad
+3C^2h_{\mu_{1}}^2+C^4h_{\mu_{1}}^4
+2+2C^2\frac{\mu_{1}^2}{\mu_{2}^2}h_{\mu_{1}}'^{2}\bigg](\mu_{1}\theta).
\end{align*}
Set $$a=\frac{\mu_{2}}{\mu_{1}}.$$ By \eqref{eq-orgn} with $i=1$,
we have
\begin{align*}
 \mathfrak{L}_{\mu_{2}}(C\widetilde{h}_{\mu_{1}})&=
 \frac{C^2}{a^2}(-(h_{\mu_{1}}''h_{\mu_{1}}^3+h_{\mu_{1}}^4))\cdot\bigg\{-(C^2-1)+
\bigg [3(a^2-1)h_{\mu_{1}}^2\\
&\qquad+C^2(a^2-1)h_{\mu_{1}}^4+\frac{2}{C^2}(a^2-1)+\frac{2}{C^2}-2\bigg]
 \frac{1}{-(h_{\mu_{1}}''h_{\mu_{1}}^3+h_{\mu_{1}}^4)}\bigg\}.
\end{align*}
By  \eqref{eq-orgn} again, we have
$$h_{\mu_{1}}''+h_{\mu_{1}}<0.$$
To prove \eqref{eq-claim}, it is equivalent to verify
\begin{align}\label{A}\begin{split}
C^2-1&\geq
\bigg[3(a^2-1)h_{\mu_{1}}^2+C^2(a^2-1)h_{\mu_{1}}^4 +\frac{2}{C^2}(a^2-1)+\frac{2}{C^2}-2\bigg]\\
&\qquad\cdot \frac{1}{-(h_{\mu_{1}}''h_{\mu_{1}}^3+h_{\mu_{1}}^4)},
\end{split}
\end{align} with $C=C_{\mu_{1},\mu_{2}}$ as in \eqref{C12}.
First, \eqref{C12} implies
$$C^{2}=1+b_{\mu_{1}}(a^2-1).$$
By $\mu_{2}\in(\mu_{1},\mu_{1}+\delta(\mu_{1}))$ and the definition of $\delta(\mu_{1}),$
we have
$$C^4\leq(\frac{9}{8})^2.$$
Then, using the definition of $b_{\mu_{1}}$ in \eqref{eq-def-b}, we get
\begin{align*}
C^2(a^2-1)h_{\mu_{1}}^4 \leq C^2(a^2-1)\max(h_{\mu_{1}}^4) 
\leq\frac{2b_{\mu_{1}}(a^2-1)}{C^2}=-(\frac{2}{C^2}-2).
\end{align*}
Hence,
$$C^2(a^2-1)h_{\mu_{1}}^4 +(\frac{2}{C^2}-2)\leq 0.$$
Now we verify
\begin{align*}
C^2-1\geq(3h_{\mu_{1}}^2+\frac{2}{C^2})(a^2-1)
 \frac{1}{-(h_{\mu_{1}}''h_{\mu_{1}}^3+h_{\mu_{1}}^4)}.
\end{align*}
Note that ${2}/{C^2}\leq 2$ and $[-(h_{\mu_{1}}''h_{\mu_{1}}^3+h_{\mu_{1}}^4)]^{-1}$ is bounded.
By the definition of $b_{\mu_{1}},$ we have
$$C^{2}-1=b_{\mu_{1}}(a^2-1)\geq (3h_{\mu_{1}}^2+\frac{2}{C^2})(a^2-1)
 \frac{1}{-(h_{\mu_{1}}''h_{\mu_{1}}^3+h_{\mu_{1}}^4)}.$$
This ends the proof of \eqref{eq-claim}.

Next, we prove the first inequality in (\ref{eq-sol-diff-cone}).
We aim to verify
$\mathfrak{L}_{\mu_{2}}(\widetilde{h}_{\mu_{1}})\geq 0$
and then proceed similarly as in the first part of the present proof.
By the earlier calculation  and (\ref{eq-orgn}), we have
\begin{align*} \mathfrak{L}_{\mu_{2}}(\widetilde{h}_{\mu_{1}})(\theta)&
=\bigg[\frac{\mu_{1}^2}{\mu_{2}^2}h_{\mu_{1}}''h_{\mu_{1}}
+\frac{\mu_{1}^2}{\mu_{2}^2}h_{\mu_{1}}''h_{\mu_{1}}^{3}
+3h_{\mu_{1}}^2+h_{\mu_{1}}^4+2+2\frac{\mu_{1}^2}{\mu_{2}^2}h_{\mu_{1}}'^{2}\bigg](\mu_{1}\theta)
\\
&=\bigg[\bigg(1-\frac{\mu_{1}^2}{\mu_{2}^2}\bigg)
(3h_{\mu_{1}}^2+h_{\mu_{1}}^4+2)\bigg](\mu_{1}\theta)\geq 0,
\end{align*}
where we used $\mu_1<\mu_2$ in the last inequality.
\end{proof}

We now compare $a_\mu$ in \eqref{u exp} for different $\mu.$ 

\begin{lemma}\label{lemma-mu} Let $\mu_1$ and $\mu_2$ be two distinct constants in $(0,1)$
and $a_{\mu_i}$ be defined as in \eqref{u exp} for $\mu=\mu_i$, $i=1, 2$. 
Then, for $\mu_1<\mu_{2}<\mu_{1}+\delta(\mu_{1}),$
\begin{align*}
\frac{\mu_2}{\mu_1}C_{\mu_1,\mu_2}^{-3}a_{\mu_1}\le 
a_{\mu_2}\le     \frac{\mu_2}{\mu_1}a_{\mu_1},
\end{align*}
where $\delta(\mu_1)$ and $C_{\mu_1,\mu_2}$ are  determined in Lemma \ref{lem-sol-diff-cone}.
\end{lemma}

\begin{proof} By \eqref{u exp}, we have 
\begin{align*}
    \tan \theta &=a_{\mu_{1}}\left(\frac{h_{\mu_{1}}(\theta)}{\cos\theta}\right)^3
    +O\bigg(\left(\frac{h_{\mu_{1}}(\theta)}{\cos\theta}\right)^4\bigg),\\
     \tan \theta' &=a_{\mu_{2}}\left(\frac{h_{\mu_{2}}(\theta')}{\cos\theta'}\right)^3
    +O\bigg(\left(\frac{h_{\mu_{2}}(\theta')}{\cos\theta'}\right)^4\bigg).
\end{align*}
Take $\theta'=\frac{\mu_2}{\mu_1}\theta.$
Then, 
\begin{align*}
\lim_{\theta\rightarrow0} 
\frac{a_{\mu_{1}}h_{\mu_{1}}^3(\theta)}{a_{\mu_{2}}h_{\mu_{2}}^3\big(\frac{\mu_2}{\mu_1}\theta\big)}
=\frac{\mu_1}{\mu_2}.
\end{align*}
By Lemma \ref{lem-sol-diff-cone}, for any $\mu_2\in (\mu_1,\mu_1+\delta(\mu_1)), $ we have
\begin{align*}
h_{\mu_{1}}(\theta)\le h_{\mu_{2}}\left(\frac{\mu_2}{\mu_1}\theta\right)
\le  C_{\mu_{1},\mu_{2}}h_{\mu_{1}}(\theta).
\end{align*}
This implies the desired result by letting $\theta\to0$. 
\end{proof}

We conclude this section with a remark on $a_\mu$. 
We note that $a_\mu$ is defined in \eqref{u exp}. In fact, it can be computed by $h_\mu$ directly
as follows: 
$$\lim_{\theta\to0}\theta^{-\frac13}h_\mu(\theta)=\frac{1}{\sqrt[3]{a_\mu}}.$$
This is implied by \eqref{g-growth rate}. 


\section{Local Asymptotic Expansions}\label{sec-Localization}

In this section, we prove that asymptotic expansions near singular boundary points up to certain orders 
are local properties. 

\begin{lemma}\label{localization}
Let $\Omega$ and $\Omega_*$ be two convex domains in $\mathbb R^2$ such that, for 
some $x_0\in \partial\Omega$ and some $R_0>0$, 
$$\Omega\bigcap B_{R_0}(x_0)=\Omega_*\bigcap B_{R_0}(x_0),$$
and that $\partial\Omega\cap B_{R_0}(x_0)$ consists of two $C^{1,1}$-curves
$\sigma_1$, $\sigma_2$ intersecting at $x_0$  
with the angle between the tangent lines of $\sigma_1$ and $\sigma_2$
given by $\mu \pi,$ for some $\mu\in(0,1).$
Suppose that $f$ and $f_*$ are solutions of \eqref{eq-MinGmain} 
for $\Omega$ and $\Omega_*$, respectively. 
Then, for some $\tau\in (0,1)$,
\begin{align}\label{u1-u2}
|f(x)-f_*(x)|\leq C f(x)\left(\frac{|x-x_0|}{r_0}\right)^{\tau}\quad\text{for any }x\in \Omega\cap B_{r_0}(x_0), 
\end{align}
where $r_0$ and $C$ are positive constants depending only 
on $R$, $\mu$ and the $C^{1,1}$-norms of $\sigma_1$ and $\sigma_2$ in $B_{R_0}(x_0)$.
\end{lemma}

\begin{proof} 
Set $\nu_i$ to be the unit inner normal vector to $\sigma_i$ at $x_0$, for $i=1,2.$
Note $\Omega\subseteq V_{x_{0}}$ since $\Omega$ is convex.
By the maximum principle, we have, for any $x\in \Omega,$
\begin{align}\label{1}
   f(x)\leq f_\mu(x)=|x-x_0| h_\mu (\theta),
\end{align}
where $f_\mu$ is the solution of \eqref{eq-MinGmain} for $\Omega= V_{x_{0}}.$
We consider two cases. 

{\it Case 1}. We first prove (\ref{u1-u2}) in the region
$\{d\geq |x-x_0|^{\frac{3}{2}}\}$.

Since both $\sigma_1$ and $\sigma_2$ are $C^{1,1},$
there exists a positive constant $R$, depending only on $R_0$, $\mu$, and 
the $C^{1,1}$-norms of $\sigma_1$ and $\sigma_2$ in $B_{R_0}(x_0)$, 
such that
$$\widetilde\Omega\equiv B_{R}(x_0+R\nu_1)\bigcap B_{R}(x_0+R\nu_2)\subseteq \Omega.$$
Let $\widetilde f$ be the solution of \eqref{eq-MinGmain} 
for $\widetilde\Omega$. The maximum principle implies 
\begin{equation}\label{eq-LowerBound}
f\ge \widetilde f\quad\text{in }\widetilde\Omega.\end{equation}
We note that the tangent cone of $\Omega$ at $x_0$ is also the tangent 
cone of $\widetilde\Omega$ at $a_0$. 
It is easy to see that $\partial B_{R}(x_0+R\nu_1)$ and $\partial B_{R}(x_0+R\nu_2)$
intersect at two points, one of which is $x_0$ and another denoted by  $q.$
A simple calculation yields
$$| x_0q|=2R\sin \frac{\mu \pi}{2}.$$
Set $L=R\sin  \frac{\mu \pi}{2}$. For convenience, we assume
\begin{align*}
  x_0=(-L,0),\quad
  q=(L,0).
\end{align*}
We consider the map $T_{L}$ introduced in \eqref{eq-T}.
Then, $T_{L}$ maps the minimal  surface
$\{(x,\widetilde f(x))\}$ in $\mathds{H}^3$ to
the minimal surface $\{(y, \widetilde{f}_\mu(y))\}$ in  $\mathds{H}^3$ and
maps conformally $\Omega_0$ 
to an infinite cone $\widetilde{V}$,
which conjugates  $V_{x_{0}}.$ 
Note
\begin{align*}
\widetilde{V}=V_{x_{0}}+\frac12{\overrightarrow{x_{0}q}}.
\end{align*}
By \eqref{eq-T} and \eqref{1}, we have
\begin{align*}
    JT_{L}|_{(x_0, 0)}=\frac{1}{2}I_{3\times3},
\end{align*}
and, for $|x-x_0|$ small,
\begin{align*}
y_1&=\frac{1}{2}(x_1+L)
+O(|x-x_{0}|^2),
\\y_2&=\left(\frac{1}{2}+O(|x-x_{0}|)\right)x_2,\end{align*}
and 
\begin{align*}
\widetilde{f}_\mu(y)=\left(\frac{1}{2}+O(|x-x_{0}|)\right)\widetilde f(x).
\end{align*}
Corollary \ref{cor-Existence} implies, for
$\theta\in [-\frac12\mu \pi+\frac{\delta\pi}{2},\frac12\mu\pi-\frac{\delta\pi}{2}],$
$$|h_\mu'(\theta)|\leq C_* \delta^{-\frac{2}{3}},$$
where $C_*$ is some positive constant depending only on $\mu.$
With $\widetilde {f}_\mu(y)=|y|h_\mu(\theta)$, we obtain
$$|\nabla \widetilde{f}_\mu|\leq C\delta^{-\frac{2}{3}} \quad\text{for }
\theta\in  \left[-\frac12\mu \pi+\frac{\delta\pi}{2},\frac12\mu\pi-\frac{\delta\pi}{2}\right].$$
If $|x-x_0|$ is small and $d \geq|x-x_0|^{\frac{3}{2}}$, then the angle $\delta$ between 
$\overline{xx_0}$ and $l_i$ is greater than $|x-x_0|^{\frac{1}{2}}/2,$ for $i=1,2$, 
where $l_i$ is the tangent line of $\sigma_i$ at $x_0$. 
By \eqref{g-growth rate} and \eqref{deriv-g-growth rate}, we have
\begin{align*}
\widetilde{f}_\mu(y)&= \widetilde{f}_\mu\left(\frac{1}{2}(x_1+L),\frac{1}{2}x_2\right)
+\delta^{-\frac23}O(|x-x_0|^{2})
\\&=\frac{1}{2}\widetilde{f}_\mu(x_1+L,x_2)
+O(|x-x_0|^{\frac{5}{3}})\\&=\frac{1}{2}|x-x_0|h_\mu(\theta)(1+O(|x-x_0|^{\frac{1}{2}})).
\end{align*}
Therefore, in the region $\{d\geq |x-x_0|^{\frac{3}{2}}\}$, we have
\begin{align*}
    \widetilde f(x)=|x-x_0|h_\mu(\theta)(1+O(|x-x_0|^{\frac{1}{2}})).
\end{align*}
By combining with \eqref{eq-LowerBound}, we get, 
for any $x$ with small $|x-x_0|$ and $d \geq|x-x_0|^{\frac{3}{2}}$, 
\begin{align}\label{eq-Lower}
f(x)\ge|x-x_0|h_\mu(\theta)(1+O(|x-x_0|^{\frac{1}{2}})).
\end{align}
By \eqref{1} and \eqref{eq-Lower}, we obtain, for such $x$,  
\begin{equation}\label{eq-Estimate-Both}
|x-x_0|h_\mu(\theta)(1+O(|x-x_0|^{\frac{1}{2}}))\le f(x)\le |x-x_0|h_\mu(\theta).\end{equation}
Similar estimates also hold for $f_*$. 
Hence, for such $x$, 
\begin{align}\label{eq-Step1}
    |f(x)-f_*(x)|\leq C_0f(x)|x-x_0|^{\frac{1}{2}}.
\end{align}

{\it Case 2.} Next, we prove (\ref{u1-u2}) in the region
$\{d\leq |x-x_0|^{\frac{3}{2}}\}$.
In the following, we assume $x_0=0$. 
We will prove there exist a constant $C$ sufficiently large 
and two constants $\alpha,r_0$ sufficiently small such that
\begin{align}\label{u1 u2}
f_*\geq f(1+f^{\alpha}-C |x|^{\alpha})
\end{align}
in 
\begin{align}\label{near bdry}
\Omega_0\equiv
\{d\leq |x|^{\frac{3}{2}}\}\bigcap\{1+f^{\alpha}-C |x|^{\alpha}\geq 0\}\bigcap B_{r_0}\bigcap\Omega.
\end{align}
Let 
\begin{equation}\label{eq-choice}C=\frac{1}{r_0^\alpha}+1,\quad \alpha=\frac{1}{100}.\end{equation}
We can take $r_{0}$ small enough such that $$\frac{1}{r_0^\alpha}\geq 2C_{0},$$
where $C_0$ is the constant as in \eqref{eq-Step1}. 
We set 
$$\widetilde{Q}(w)=w(1+|\nabla w|^2)\left(\Delta w-\frac{w_iw_jw_{ij}}{1+|\nabla w|^2}+\frac{n}{w}\right),$$
or 
\begin{align*}
\widetilde{Q}(w)=w\Delta w+w(|\nabla w|^2\Delta w-w_iw_jw_{ij})
+n+n|\nabla w|^2.
\end{align*}
We claim 
\begin{equation}\label{eq-domain}
\widetilde{Q}(f(1+f^{\alpha}-C |x|^{\alpha}))\geq 0=\widetilde{Q}(f_*)\quad\text{in }\Omega_0\end{equation}
and 
\begin{equation}\label{eq-boundary}
f_*\geq f(1+f^{\alpha}-C |x|^{\alpha})\quad\text{on }\partial\Omega_0.\end{equation}
Then, the maximum principle implies \eqref{u1 u2}. 

We first prove \eqref{eq-boundary}. Since $\Omega$ is convex, 
$\Omega$ is in the tangent cone of $\Omega$ at 0. 
By \eqref{1}, we have
\begin{equation*}\label{eq-Estimate-f}
f\leq C_{\mu}|x|\left(\sin\frac{\theta}{\mu}\right)^{\frac{1}{3}}.\end{equation*}
A simple geometric argument 
yields 
\begin{align*}
f\leq C_{\mu}|x|\cdot \left(\frac{d+C|x|^2}{|x|}\right)^{\frac{1}{3}},
\end{align*}
and hence,
\begin{align}\label{r divide u}
\frac{f}{|x|}\leq C_{\mu}|x|^{\frac{1}{6}}\quad\text{in }\Omega_0.
\end{align}
Note that, for $r_{0}$ small,
\begin{align}\label{h-size}
f^{\alpha}\leq \frac{|x|^{\alpha}}{100}.
\end{align}
Therefore, by \eqref{eq-Step1} and \eqref{h-size}, we have \eqref{eq-boundary}. 
We note that we need to discuss $|x|=r_0$ and $d=|x|^{\frac32}$ separately.

Now we proceed to prove \eqref{eq-domain}. We will do this for general $n$
under the conditions \eqref{r divide u} and \eqref{h-size}.  Set 
\begin{align*}
h_0&=1+f^{\alpha}-C|x|^{\alpha},\\
h&=1+(1+\alpha)f^{\alpha}-C|x|^{\alpha},\end{align*}
and 
$$g=fh_0.$$
Then, 
\begin{align*}
g_i=f_ih-C\alpha |x|^{\alpha-2}fx_i,\end{align*}
and 
\begin{align*}
g_{ij}&=f_{ij}h+\alpha(1+\alpha)f^{\alpha-1}f_if_j-C\alpha|x|^{\alpha-2}(f_ix_j+f_jx_i)\\
&\qquad -C\alpha(\alpha-2)|x|^{\alpha-4}fx_ix_j-C\alpha|x|^{\alpha-2}f\delta_{ij}.
\end{align*}
Hence, 
$$|\nabla g|^2=|\nabla f|^2h^2-2C\alpha|x|^{\alpha-2}fh(x\cdot\nabla f)
+C^2\alpha^2|x|^{2\alpha-2}f^2,$$
and 
$$\Delta g=\Delta fh+\alpha(1+\alpha)f^{\alpha-1}|\nabla f|^2-2C\alpha|x|^{\alpha-2}(x\cdot\nabla f)
-C\alpha(\alpha-2+n)|x|^{\alpha-2}f.$$
Next, a straightforward calculation yields 
\begin{align*}
|\nabla g|^2\Delta g-g_ig_jg_{ij}
&=(|\nabla f|^2\Delta f-f_if_jf_{ij})h^3\\
&\qquad -2C\alpha|x|^{\alpha-2}fh^2[(x\cdot\nabla f)\Delta f-x_if_jf_{ij}]\\
&\qquad +C^2\alpha^2|x|^{2\alpha-4}f^2h[|x|^2\Delta f-x_ix_jf_{ij}]\\
&\qquad +C\alpha(2-\alpha)|x|^{\alpha-4}fh^2[|x|^2|\nabla f|^2-(x\cdot \nabla f)^2]\\
&\qquad +C^2\alpha^3(1+\alpha)|x|^{2\alpha-4}f^{\alpha+1}[|x|^2|\nabla f|^2-(x\cdot \nabla f)^2]\\
&\qquad -2C^2\alpha^2|x|^{2\alpha-4}fh[|x|^2|\nabla f|^2-(x\cdot \nabla f)^2]\\
&\qquad -C\alpha(n-1)|x|^{\alpha-2}fh^2|\nabla f|^2\\
&\qquad +2C^2\alpha^2(n-1)|x|^{2\alpha-4}f^2h(x\cdot\nabla f)
-C^3\alpha^3(n-1)|x|^{3\alpha-4}f^3.
\end{align*}
Then, we can write $\widetilde Q(g)$ as 
$$\widetilde Q(g)=I+II+III,$$
where 
\begin{align*}
I&=f\Delta fh_0h+f(|\nabla f|^2\Delta f-f_if_jf_{ij})h^3h_0+
n+n|\nabla f|^2h^2,\\
II
&=-2C\alpha|x|^{\alpha-2}f^2h^2h_0[(x\cdot\nabla f)\Delta f-x_if_jf_{ij}]\\
&\qquad +C^2\alpha^2|x|^{2\alpha-4}f^3hh_0[|x|^2\Delta f-x_ix_jf_{ij}]\\
&\qquad +2C^2\alpha^2|x|^{2\alpha-4}f^2hh_0[(x\cdot \nabla f)^2-|x|^2|\nabla f|^2]\\
&\qquad -C\alpha(n-1)|x|^{\alpha-2}f^2h^2h_0|\nabla f|^2\\
&\qquad +2C^2\alpha^2(n-1)|x|^{2\alpha-4}f^3hh_0(x\cdot\nabla f)\\
&\qquad-2C\alpha|x|^{\alpha-2}fh_0(x\cdot \nabla f)
-2nC\alpha|x|^{\alpha-2}fh(x\cdot \nabla f)\\
&\qquad-C^3\alpha^3(n-1)|x|^{3\alpha-4}f^4h_0
-C\alpha(\alpha-2+n)|x|^{\alpha-2}f^2h_0, 
\end{align*}
and 
\begin{align*}
III
&=C\alpha(2-\alpha)|x|^{\alpha-4}f^2h^2h_0[|x|^2|\nabla f|^2-(x\cdot \nabla f)^2]\\
&\qquad +C^2\alpha^3(1+\alpha)|x|^{2\alpha-4}f^{\alpha+2}h_0[|x|^2|\nabla f|^2-(x\cdot \nabla f)^2]\\
&\qquad +\alpha(1+\alpha)f^\alpha h_0|\nabla f|^2+nC^2\alpha^2|x|^{2\alpha-2}f^2.
\end{align*}
First, we note that $III\ge 0$. Next by $\widetilde{Q}(f)=0,$ we have, 
\begin{align}\label{u12=0}
 f\Delta f+f(|\nabla f|^2\Delta f-f_if_jf_{ij})+n+n|\nabla f|^2=0.
\end{align} Then, 
\begin{align*}
I&= f\Delta fh_0h+f(|\nabla f|^2\Delta f-f_if_jf_{ij})h^3h_0+
   n+n|\nabla f|^2h^2\\
&=-(hh_0-h^3h_0)f(|\nabla f|^2\Delta f-f_if_jf_{ij})
+n(h^2-hh_0)|\nabla f|^2+n(1-hh_0)\\
&=-hh_0(1-h^2)f(|\nabla f|^2\Delta f-f_if_jf_{ij})
+n\alpha f^\alpha h|\nabla f|^2+n(1-hh_0).
\end{align*}
By \eqref{h-size}, we have 
$$
0\leq h_0<h\leq 1\quad\text{in }\Omega_0,$$
and 
\begin{align*}1-h^2\geq \frac{99}{100}C|x|^\alpha,\quad
1-hh_0\geq \frac{99}{100}C|x|^\alpha\quad\text{in }\Omega_0.
\end{align*}
We note that $\widetilde{Q}(w)$ is invariant under orthogonal transforms. 
Fix a point $p\in\Omega_0$ and assume, by a rotation, 
that $f_{ij}(p)=0$ for $i\neq j$. In the following, we calculate  $\widetilde{Q}(g)$ at $p$. 
First, 
\begin{align*}
I=-hh_0(1-h^2)f\sum_{i=1}^n(|\nabla f|^2-f_i^2)f_{ii}
+n\alpha f^\alpha h|\nabla f|^2+n(1-hh_0).
\end{align*}
Since $f$ is concave by Theorem 3.1 \cite{HanShenWang}, then $f_{ii}\le 0$ and hence
\begin{align}\label{eq-Estimate-I}
I\ge\frac{99}{100}C|x|^\alpha hh_0f\sum_{i=1}^n\sum_{k\neq i}f_k^2|f_{ii}|
+n\alpha f^\alpha h|\nabla f|^2+\frac{99}{100}nC|x|^{\alpha}. 
\end{align}
We now consider terms in $II$. For illustrations, we consider 
the following three terms: 
\begin{align*}II_1&=-2C\alpha|x|^{\alpha-2}f^2h^2h_0[(x\cdot\nabla f)\Delta f-x_if_jf_{ij}],\\
II_2&=-C\alpha(n-1)|x|^{\alpha-2}f^2h^2h_0|\nabla f|^2,\\
II_3&=-C^3\alpha^3(n-1)|x|^{3\alpha-4}f^4h_0.\end{align*}
For $II_1$, we write 
\begin{align*}II_1=
-2C\alpha\left(\frac{f}{|x|}\right)^{1-\alpha}f^{1+\alpha}
h^2h_0\sum_{i=1}^n\sum_{k\neq i}\frac{x_k}{|x|}f_kf_{ii}.\end{align*}
By \eqref{r divide u}, we have 
$$C\left(\frac{f}{|x|}\right)^{1-\alpha}\le 
\left(\frac{1}{r_0^\alpha}+1\right)(C_\mu|x|^{\frac16})^{1-\alpha}\le \frac{1}{100}.$$
Hence, 
\begin{align*}
|II_1| &\leq 
\frac{\alpha}{100}f^{1+\alpha} h^2h_0\sum_{i=1}^n\sum_{k\neq i}f_k^2|f_{ii}|
+\frac{\alpha}{100} f^\alpha  h^2h_0 f\sum_{i=1}^n|f_{ii}|\\
&\leq
\frac{\alpha}{100}f^{\alpha} h^2h_0f\sum_{i=1}^n\sum_{k\neq i}f_k^2|f_{ii}|
+\frac{n\alpha}{100} f^\alpha  h  (1+|\nabla f|^2),
\end{align*}
where we used \eqref{u12=0} with $f_{ii}\le 0$. 
For $II_2$, we write 
$$II_2=-C\alpha(n-1)\left(\frac{f}{|x|}\right)^{2-\alpha}f^\alpha h^2h_0|\nabla f|^2.$$
Then, 
$$|II_2|\le (n-1)\alpha \left(\frac{1}{r_0^\alpha}+1\right) (C_\mu r_0)^{2-\alpha}f^\alpha h^2h_0|\nabla f|^2
\le \frac{\alpha}{100} f^\alpha h |\nabla f|^2.$$
For $II_3$, we write 
$$II_3
=-C^3\alpha^3(n-1)\left(\frac{f}{|x|}\right)^{4-3\alpha}f^{3\alpha}h_0.$$
Then, 
\begin{align*}
|II_3| \le \left(\frac{1}{r_0^\alpha}+1\right)^3\alpha^3(n-1)(C_\mu |x|^{\frac16})^{4-3\alpha}f^{\alpha}h_0
\leq \frac{\alpha}{100}f^{\alpha}h.
\end{align*}
We can consider other terms in $II$ similarly. Therefore, with \eqref{eq-Estimate-I}, we obtain 
$\widetilde Q(g)\ge 0$ at $p\in \Omega_0$. 
Since $p$ is arbitrary, we have \eqref{eq-domain}. 
\end{proof}

\section{Singular Points with Positive Curvatures}\label{sec-C-2,alpha-boundary}

Let $\Omega$ be a bounded convex domain in $\mathbb R^2$ and, for some
$x_0\in\partial\Omega$ and $R>0$, let 
$\partial\Omega\cap B_R(x_0)$  consist of two $C^{2,\alpha}$-curves
$\sigma_1$ and $\sigma_2$
intersecting at the origin at an angle $\mu\pi$, for some
constants $\alpha, \mu\in (0,1)$. Assume the curvature $\kappa_i$ 
of $\sigma_i$ at $x_0$ is positive and denote by $R_i=1/\kappa_i$. Set 
$$\Omega_{x_0,\mu,\kappa_{1},\kappa_{2}}=
B_{R_1} (x_0+R_1\nu_1)\bigcap B_{R_2} (x_0+R_2\nu_2),$$ 
where $\nu_i$ is the unit inner normal vector of $\sigma_i$ at $x_0$. 
Then, any $x\in\Omega_{x_0,\mu,\kappa_{1},\kappa_{2}}$ near $x_0$
is uniquely determined by $d_1, d_2$, where $d_i(x)$ is the distance 
from $x$ to $\partial B_{R_i}(R_i\nu_i)$.
With such a  one-to-one correspondence between $x\in\Omega_{x_0,\mu,\kappa_{1},\kappa_{2}}$ 
near $0$ and $(d_1, d_2)$ with $d_1 > 0$ and
$d_2 > 0$ small, we rewrite the solution of \eqref{eq-MinGmain} 
for $\Omega=\Omega_{x_0,\mu,\kappa_{1},\kappa_{2}}$ as 
\begin{equation}\label{v-d1-d2} f_{x_0,\mu,R_1,R_2}(d_1,d_2).\end{equation}

We prove the following result in this section.

\begin{theorem}\label{thrm-C-2,alpha-expansion}
Let $\Omega$ be a bounded convex domain in $\mathbb R^2$ and, for some
$x_0\in\partial\Omega$ and $R>0$, let 
$\partial\Omega\cap B_R$  consist of two $C^{2,\alpha}$-curves
$\sigma_1$ and $\sigma_2$
intersecting at $x_0$ at an angle $\mu\pi$, for some
constants $\alpha, \mu\in (0,1)$. Assume the curvature $\kappa_i$ 
of $\sigma_i$ at $x_0$ is positive. 
Suppose  $f\in C(\bar\Omega)\cap C^\infty (\Omega)$ is the solution of 
\eqref{eq-MinGmain} in $\Omega$.
Then, for any $x\in \Omega$ close to $x_0$,
\begin{align}
\label{main-estimate2a} \left|f(x)-f_{x_0,\mu,R_1,R_2}(d_1,d_2)\right|\leq Cf(x)|x-x_0|^{\beta},
\end{align}
where $d_i$ is the distance to $\sigma_i$, 
$f_{x_0,\mu,R_{1},R_{2}}$ is the solution of \eqref{eq-MinGmain} in 
$\Omega_{x_0,\mu,\kappa_{1},\kappa_{2}}$ in terms of $d_1$ and $d_2$, 
$\beta$ is a constant in $(0,\alpha/2]$, and $C$ is a positive constant depending only on 
$R$, $\mu$, $\alpha$, and
the $C^{2,\alpha}$-norms of $\sigma_1$ and $\sigma_2$ in $B_R(x_0)$.
\end{theorem}

\begin{proof}
Without loss of generality, we assume $d_1\leq d_2$. Then, 
$$d_2\leq|x-x_0|\leq Cd_2.$$
We consider two cases. 

{\it Case 1}. We first consider the case 
$d_1\geq |x-x_0|^{\frac{3}{2}}$. 
By Case 1 in the proof of Lemma \ref{localization}, specifically \eqref{eq-Estimate-Both}, 
we have
\begin{align}\label{eq-Estimate}
    f(x)=f_\mu(x)(1+O(|x-x_0|^{\frac{1}{2}})), 
\end{align}
where $f_\mu(x)$ is the corresponding solution of \eqref{eq-MinGmain} 
on the tangent cone $V_{x_0}$ of $\Omega$ at $x_0.$
Let $x^*$ be the unique point in $\Omega_{x_0, \mu, \kappa_1, \kappa_2}$ 
determined by $d_1,d_2$.  Then, 
\begin{align*}
   |x^*-x|\leq C|x-x_0|^{2+\alpha}.
\end{align*}
Note that $V_{x_0}$ is also the tangent cone of 
$\Omega_{x_0, \mu, \kappa_1, \kappa_2}$ at $x_0$. Hence, 
\begin{align*}
    f_{x_0,\mu,R_1,R_2}(d_1,d_2)=f_\mu(x^*)(1+O(|x^*-x_0|^{\frac{1}{2}})).
\end{align*}
By the mean value theorem and $\theta\ge |x-x_0|^{\frac12}$, we get 
\begin{align*}
    |f_\mu(x)-f_\mu(x^*)|\leq C|x-x_0|^{-\frac{1}{3}}|x-x_0|^{2+\alpha}
    \leq Cf_\mu(x)|x-x_0|^{\frac{1}{2}+\alpha}.
\end{align*}
Therefore,
\begin{align*}
    f(x)=f_{x_0,\mu,R_1,R_2}(d_1,d_2)(1+O(|x-x_0|^{\frac{1}{2}})).
\end{align*}

{\it Case 2.} We consider 
$d_1\leq |x-x_0|^{\frac{3}{2}}$.
Denote by $p_i$ the point on $\sigma_i$ closest to $x$ and by 
$\nu_{p_i}$ the unit inner normal vector to $\sigma_i$ at $p_i.$
Set, for $i=1,2,$
\begin{align*}
    \widetilde{R}_i=R_i+|x-x_0|^{\frac{\alpha}{2}},\quad
   \widehat{ R}_i=R_i-|x-x_0|^{\frac{\alpha}{2}},
\end{align*}
and
\begin{align*}
    \widetilde{\Omega}=\bigcap_{i=1}^2 B_{\widetilde{R}_i}(p_i+\widetilde{R}_i\nu_{p_i}),\quad
     \widehat{\Omega}=\bigcap_{i=1}^2 B_{\widehat{ R}_i}(p_i+\widehat{ R}_i\nu_{p_i}).
\end{align*}
Let $\widetilde{f},\widehat{f}$ be the solution of \eqref{eq-MinGmain} 
for $\Omega=\widetilde{\Omega},\widehat{\Omega},$ respectively.

For $|x-x_0|$ small, it is straightforward to verify
\begin{align*}
    \Omega\bigcap B_{C_0|x-x_0|^{\frac{1}{2}}}(x_0)\subset \widetilde{\Omega},\quad 
    \widehat{\Omega}\bigcap B_{C_0|x-x_0|^{\frac{1}{2}}}(x_0)\subset \Omega.
\end{align*}
If $|x-x_0|$ is small, $\partial B_{\widetilde{R}_1}(p_1+\widetilde{R}_1\nu_{p_1})
$ and $\partial B_{\widetilde{R}_2}(p_2+\widetilde{R}_2\nu_{p_2})$ intersect 
at points $\widetilde{p},\widetilde{q}.$ 
Without loss of generality, we denote by $\widetilde{p}$ the point closer to $x_0$.
Similarly, $\partial B_{\widehat{R}_1}(p_1+\widehat{R}_1\nu_{p_1})
$ and $\partial B_{\widehat{R}_2}(p_2+\widehat{R}_1\nu_{p_2})$ intersect at points
$\widehat{p},\widehat{q}.$ We denote by  $\widehat{p}$ the point closer to $x_0$.
It is easy to verify
\begin{align*}
    |\widetilde{p}-x_0|\leq C|x-x_0|^2,\quad
    |\widehat{p}-x_0|\leq C|x-x_0|^2.
\end{align*}
Hence,\begin{align*}
    \widehat{\Omega}\bigcap B_{C_0'|x-x_0|^{\frac{1}{2}}}(\widehat{p})\subset \Omega.
\end{align*}
Let $f',\widehat{f}'$ be the solution of \eqref{eq-MinGmain} for 
$\Omega\bigcap B_{C_0|x-x_0|^{\frac{1}{2}}}(x_0)$, 
$\widehat{\Omega}\bigcap B_{C_0'|x-x_0|^{\frac{1}{2}}}(\widehat{p}),$ respectively.
By Lemma \ref{localization},
we have 
$$|f'(x)-f(x)|\leq 
Cf(x)\left(\frac{|x-x_0|}{C_0|x-x_0|^{\frac{1}{2}}}\right)^{\tau}\leq Cf(x)|x-x_0|^{\frac{\tau}{2}},$$
and
$$|\widehat{f}'(x)-\widehat{f}(x)|\leq
\widehat{f}(x)\left(\frac{|x-\widehat{p}|}{C_0|x- x_0|^{\frac{1}{2}}}\right)^{\tau}
\leq C\widehat{f}(x)|x-x_0|^{\frac{\tau}{2}},$$
where we took $r_0=C_0|x-x_0|^{\frac12}$ in \eqref{u1-u2}. 
By the maximum principle, we have
$$f'(x)\leq\widetilde{f}(x),\quad \widehat{f}'(x)\leq f(x).$$
Hence, 
\begin{align*}
 \widehat{f}(x)(1-C|x-x_0|^{\frac{\tau}{2}})&\leq f(x),\\
  f(x)(1-C|x-x_0|^{\frac{\tau}{2}})&\leq \widetilde{f}(x).
\end{align*}
Therefore,
\begin{align}\label{eq-comparison}
   \widehat{f}(x)(1-C|x-x_0|^{\frac{\tau}{2}})\leq f(x)\leq \widetilde{f}(x)
   (1+C|x-x_0|^{\frac{\tau}{2}}).
\end{align}

Let $\widetilde{\mu}\pi$ be the openning angle of the tangent cone of 
$\widetilde{\Omega}$ at $\widetilde{p}$ and 
$\widehat{\mu}\pi$ be the openning angle of the tangent cone of 
$\widehat{\Omega}$ at $\widehat{p}.$ It is easy to check that
\begin{align}\label{eq-Angles}
    \begin{split}
      |\widetilde{\mu}-\mu|  \leq C|x-x_0|^{1+\frac{\alpha}{2}},\quad
          |\widehat{\mu}-\mu|  \leq C|x-x_0|^{1+\frac{\alpha}{2}}.
     \end{split}
\end{align}
We note that $\partial B_{R_1}(x_0+R_1\nu_1)$ and $\partial B_{R_2}(x_0+R_2\nu_2)$
intersect at two points, one of which is $x_0$ and another denoted by  $q.$
By Lemma 6.1 in \cite{HanShen2} or calculating directly, we have
$$|x_0q|=\frac{2R_1R_2\sin(\pi-\mu\pi)}{\sqrt{R_{1}^{2}+R_{2}^{2}-2R_1R_2\cos(\pi-\mu\pi)}},$$
and similar formulas for $|\widehat{p}\widehat{q}|$ and $|\widetilde{p}\widetilde{q}|$. 
Hence, 
\begin{align*}
    ||\widehat{p}\widehat{q}|-|x_0q||& \leq C|x-x_0|^{\frac{\alpha}{2}},\\
    ||\widetilde{p}\widetilde{q}|-|x_0q||& \leq C|x-x_0|^{\frac{\alpha}{2}}.
\end{align*}
We also note that 
in $\widetilde \Omega$, the distance of $x$ to $\partial B_{\widetilde{R}_i}(x_0+\widetilde R_{i}\nu_{p_i})$ 
is $d_i$ for $i=1,2$ and that 
in $\widehat \Omega$, the distance of $x$ to $\partial B_{\widehat{R}_i}(x_0+\widehat R_{i}\nu_{p_i})$ 
is $d_i$ for $i=1,2$. 
Hence, $$\widetilde{f}(x)=f_{x_0,\widetilde{\mu},\widetilde{R}_1,\widetilde{R}_2}(d_1,d_2), 
\quad 
\widehat{f}(x)=f_{x_0,\widehat{\mu},\widehat{R}_1,\widehat{R}_2}(d_1,d_2).$$
Next, we will prove, for some constant $\gamma,$
\begin{align}\label{eq-Claim-tilde}
\widetilde f(x)= f_{x_0, \mu,R_1,R_2}(d_1,d_2)
 (1+O(|x-x_0|^\gamma)),
\end{align}
and 
\begin{align}\label{eq-Claim-hat}
\widehat f(x)= f_{x_0,\mu,R_1,R_2}(d_1,d_2)
 (1+O(|x-x_0|^\gamma)).
\end{align}
We have the desired result by combining \eqref{eq-comparison}, \eqref{eq-Claim-tilde}, 
and \eqref{eq-Claim-hat}. 

Set $\widetilde L=|\widetilde{p}\widetilde{q}|/2$. 
By a translation and a rotation, we assume
\begin{align*}
    \widetilde{p}=(-\widetilde L,0),\quad
    \widetilde{q}=(\widetilde L,0).
\end{align*}
Then, $T_{\widetilde L}|_{\{x_3=0\}}$ 
transforms $\widetilde{\Omega}$ conformally to an infinite cone $V_{\widetilde{\mu}}$ 
and $T_{\widetilde L}$
transforms the minimal graph 
$\{(\widetilde{x}_1,\widetilde{f}(\widetilde{x}))\}$ 
with the asymptotic boundary $\partial\widetilde\Omega$
to the minimal graph 
$\{(\widetilde{y},f_{\widetilde{\mu}}(\widetilde{y}))\}$ 
with the asymptotic boundary $\partial V_{\widetilde{\mu}}.$ 
With $x=(x_1,x_2)$, set $y=(y_1, y_2)$ such that 
$(y,f_{\widetilde{\mu}}(y))=T_{\widetilde L}(x,\widetilde{f}(x))$. 
For brevity, set 
$$T_{\widetilde L, 0}=T_{\widetilde L}|_{\{x_3=0\}}.$$
We have
$$\widetilde{ f}(x)\leq C|x-x_0|\left(\frac{|x-x_0|^{\frac{3}{2}}}{|x-x_0|}\right)^{\frac{1}{3}}\\
   \leq C|x-x_0|^{\frac{7}{6}}.$$
Moreover, 
\begin{align*}
    JT_{\widetilde L}|_{(\widetilde p, 0)}=\frac{1}{2}I_{3\times3},
\end{align*}
and 
\begin{align*}
   \widetilde{f}_\mu(y)
   &=\frac{2\widetilde L^2\widetilde{f}(x)}{(
   x_1-\widetilde L)^2+x_2^2+\widetilde{f}^2(x)}
   =\left(\frac{1}{2}+O(|x-x_0|)\right)\widetilde{f}(x),\\
    (y_1,y_2)&=\frac{(
   x_1-\widetilde L)^2+x_2^2}{(
   x_1-\widetilde L)^2+x_2^2+\widetilde{f}^2(x)}
   T_{\widetilde L,0}(x_1,x_2)
   -\frac{\widetilde L(\widetilde{f}^2(x),0)}{(
   x_1-\widetilde L)^2+x_2^2+\widetilde{f}^2(x)}.
\end{align*}
Set 
$$\widetilde{l}_1=T_{\widetilde L,0}
   (\partial B_{\widetilde{R}_1}(p_1+\widetilde{R}_1\nu_{p_1})),$$
and write, for a unique vector $e_{\widetilde{l}_1}$, 
$$\partial V_{\widetilde{\mu}}\bigcap\widetilde{l}_1=\{te_{\widetilde{l}_1}|\ t\geq 0\}.$$
By the definition of $d_1$, we have 
\begin{align*}
    \mathrm{dist}(x,\partial B_{\widetilde{R}_1}(p_1+\widetilde{R}_1\nu_{p_1}))=
d_1.
\end{align*}
Hence, 
\begin{align*}
    \mathrm{dist}(T_{\widetilde L,0}(x),
   \widetilde{l}_1)=
   \left(\frac{1}{2}+O(|x-x_0|)\right)d_1.
\end{align*}
A simple geometric argument yields 
\begin{align*}
    \mathrm{dist}(y,\widetilde{l}_1)&=(1+O(|x-x_0|))\mathrm{dist}(
    T_{\widetilde L,0}(x),
   \widetilde{ l}_1)\\ 
   &\qquad-\frac{\widetilde L\widetilde{f}^2}{(
   x_1-\widetilde L)^2+x_2^2+\widetilde{f}^2}
   \sin\angle(\frac{\overrightarrow{\widetilde{p}\widetilde{q}}}{|\widetilde{p}\widetilde{q}|},
   e_{\widetilde{l}_1} )\\&=
  \left (\frac{1}{2}+O(|x-x_0|)\right)d_1
   -(1+O(|x-x_0|))\frac{\widetilde{f}^2}{4\widetilde{R}_1},
\end{align*}
or 
$$(1+O(|x-x_0|))d_1=(2+O(|x-x_0|))\frac{\widetilde{f}^2}{4\widetilde{R}_1}
+2\mathrm{dist}(y,\widetilde{l}_1).$$
Write $\langle y,e_{\widetilde{l}_1}\rangle=|y|\cos\theta$. By \eqref{u exp}, we have
\begin{align*}
\frac{
  \mathrm{dist}(y,\widetilde{l}_1)}{\langle y,e_{\widetilde{l}_1}\rangle}
=a_{\widetilde{\mu}}
\bigg(\frac{f_{\widetilde{\mu}}}{\langle y,e_{\widetilde{l}_1}\rangle}\bigg)^3
  +O\bigg(\bigg(\frac{
  f_{\widetilde{\mu}}}{\langle y,e_{\widetilde{l}_1}\rangle}\bigg)^4\bigg).
\end{align*}
We point out that the left-hand side is simply 
$\tan \theta$. The presence of the factor $\langle y,e_{\widetilde{l}_1}\rangle$ in the right-hand side 
is due to a scaling since \eqref{u exp} is expanded at $(r,\theta)=(1,0)$. 
Note that
\begin{align*}
\langle y,e_{\widetilde{l}_1}\rangle=|y|\cos\theta
= (1+O(|x-x_0|)) |y|
=\left(\frac{1}{2}+O(|x-x_0|)\right)|x-x_0|.
\end{align*}
Hence,
\begin{align*}
    (1+O(|x-x_0|))d_1&=\frac{1}{4\widetilde{R}_1}(2+O(|x-x_0|))\widetilde{f}^2\\
    &\qquad+[a_{\widetilde{\mu}}+O(|x-x_0|^{\frac{1}{6}})](1+O(|x-x_0|))\frac{\widetilde{f}^3}{|x-x_0|^2},
\end{align*}
where we substituted $f_{\widetilde\mu}$ by $\widetilde f$. 
By Lemma \ref{lemma-mu}, we have
\begin{align*}
    |a_{\widetilde{\mu}}-a_{\mu}|\leq C|x-x_0|^{1+\frac{\alpha}{2}},
\end{align*}
and hence
\begin{align*}
    (1+O(|x-x_0|))d_1&=\frac{1}{4R_1}(2+O(|x-x_0|)) (1+O(|x-x_0|)^{\frac{\alpha}{2}})\widetilde{f}^2(x)
   \\
    &\qquad+
    [a_{\mu}+O(|x-x_0|^{\frac{1}{6}})]\frac{\widetilde{f}^3(x)}{|x-x_0|^2},
\end{align*}
where we used the relation between $R_1$ and $\widetilde{R}_1.$
A similar argument holds for $f_{x_0,\mu,R_1,R_2}$ as defined in 
$\Omega=\Omega_{x_0,\mu,\kappa_{1},\kappa_{2}}$. Then,
\begin{align*}
    (1+O(|x-x_0|))d_1&=\frac{1}{4R_1}(2+O(|x-x_0|))f_{x_0,\mu,R_1,R_2}(d_1,d_2)^2\\
    &\qquad+
    [a_{\mu}+O(|x-x_0|^{\frac{1}{6}})]\frac{f_{x_0,\mu,R_1,R_2}(d_1,d_2)^3}{|x-x_0|^2}.
\end{align*}
Hence,
\begin{align*}
    \widetilde{f}(x)=f_{x_0,\mu,R_1,R_2}(d_1,d_2)(1+O(|x-x_0|^{\gamma})),
\end{align*}
where $\gamma= \min\{\frac{\alpha}{2},\frac{1}{6}\}.$ This implies \eqref{eq-Claim-tilde}. 
We can prove \eqref{eq-Claim-hat} similarly.
\end{proof}

\begin{remark}\label{rem-Optimal} If $\alpha\le \min\{\tau, 1/3\}$, then we can take $\beta=\alpha/2$. 
See \eqref{eq-comparison}, \eqref{eq-Claim-tilde}, 
and \eqref{eq-Claim-hat}. 
\end{remark}

We are ready to prove Theorem \ref{main theorem2}.

\begin{proof}[Proof of Theorem \ref{main theorem2}]
We adopt the notations from Theorem \ref{thrm-C-2,alpha-expansion} and its proof.
Without loss of generality, we assume $x_0$ is
the origin. 
For any $x$ sufficiently small, we define
$$T_{\{\sigma_i\}}x=(d_1(x), d_2(x)),$$ where $d_i(x)$ is the signed distance 
from $x$ to $\sigma_i$ with respect to $\nu_i$, $i=1, 2$, positive if $x$ is on the side of 
$\nu_i$ and negative if on another side. Refer to \cite{HanShen2} for details.
We emphasize that $T_{\{\sigma_i\}}$ is defined in a full neighborhood of the origin 
instead of only in $\Omega$
and that the signed distance is used instead of its absolute value. 
Then, $T_{\{\sigma_i\}}$ is $C^{2,\alpha}$ near the origin and
its Jacobi matrix at the origin is nonsingular by the linear independence of
$\nu_1$ and $\nu_2$. Therefore, $T_{\{\sigma_i\}}$ is a $C^{2,\alpha}$-diffeormorphism
in a neighborhood of the origin. We have a similar result for 
$T_{\{\partial B_{R_i} (R_i\nu_i)\}}$, with $\partial B_{R_i} (R_i\nu_i)$ replacing
$\sigma_i$, $i=1,2$.
Then, the map $T=T^{-1}_{\{\partial B_{R_i} (R_i\nu_i)\}}\circ T_{\{\sigma_i\}}$ 
is a $C^{2,\alpha}$-diffeomorphism
near the origin and has the property that the signed distance from $x$ to $\sigma_i$ is the same as that from $Tx$
to $\partial B_{R_i} (R_i\nu_i)$, for $i=1, 2$.
Therefore, Theorem \ref{main theorem2} follows from Theorem \ref{thrm-C-2,alpha-expansion}. 
\end{proof}

\end{document}